\documentclass[preprint,12pt]{imsart}

\usepackage{amsthm,amsmath,natbib}
\usepackage{amsfonts}
\usepackage{amssymb}

\RequirePackage[colorlinks,citecolor=blue,urlcolor=blue]{hyperref}

\usepackage{tikz-cd}

\usepackage[a4paper,includeheadfoot,margin=2.54cm]{geometry}

\startlocaldefs
\newcommand{\vekk}[1]{}

\usepackage{amsthm}

\newtheorem{Lemma}{Lemma}

 % ... 

\newcommand{\cF}{{\cal F}} 
\newcommand{\cE}{{\cal E}} 

\newcommand{\pr}{\operatorname{\text{P}}} 
\newcommand{\E}{{\operatorname{\text{E}}}} 

\newcommand{\RealN}{{\mbox{$\mathbb R$}}} 
\newcommand{\st}{\operatornamewithlimits{{\mid}}}
 
\newcommand{\normvar}{\mathop{\bf N}}

\newcommand{\nc}{\newcommand}
\nc{\beqns}{\begin{eqnarray*}}
\nc{\eeqns}{\end{eqnarray*}}
\nc{\beqn}{\begin{eqnarray}}
\nc{\eeqn}{\end{eqnarray}}
\nc{\beq}{\begin{equation}}
\nc{\eeq}{\end{equation}}

\newcommand{\bes}{ \begin{equation} \begin{split} } 
\newcommand{\ees}{ \end{split} \end{equation} } 

\newcommand{\be}[1]{\begin{equation}\label{#1}} 
\newcommand{\ee}{\end{equation}}

\newcommand{\into}{\mbox{$\: \rightarrow \:$}}

\newcommand{\norm}[1]{\left \| #1\right\|}

\theoremstyle{plain}

\theoremstyle{remark}

\endlocaldefs

\begin{document}

\begin{frontmatter}

  \title{Optimal Learning\\ from the Doob-Dynkin lemma}
  
\runtitle{The Doob-Dynkin lemma}

\begin{aug}
\author{\fnms{Gunnar} \snm{Taraldsen}\corref{}\ead[label=e1]{Gunnar.Taraldsen@ntnu.no}}

\address{Trondheim, Norway.
\printead{e1}}

\runauthor{Taraldsen}

\affiliation{Norwegian University of Science and Technology}

\end{aug}

\begin{abstract}
  The Doob-Dynkin Lemma gives conditions on two functions
  $X$ and $Y$ that ensure existence of a function $\phi$
  so that $X = \phi \circ Y$.
  This communication proves different versions of the Doob-Dynkin Lemma,
  and shows how it is related to optimal statistical learning algorithms.
\end{abstract}

\begin{keyword}
\kwd{Improper prior}
\kwd{Descriptive set theory}
\kwd{Conditional Monte Carlo}
\kwd{Fiducial}
\kwd{Machine learning}
\kwd{Complex data}
\end{keyword}




\tableofcontents

\end{frontmatter}

\section{Introduction}
\label{sIntro}

This note is motivated by the commutative diagram
\begin{equation}
\label{eqComm}
{\Large
\begin{tikzcd}[row sep=tiny]
& \Omega_X  \\
\Omega \arrow[ur, "X"] \arrow[dr, "Y"'] & \\
& \Omega_Y \arrow[uu, "\phi"']
\end{tikzcd}
}
\end{equation}
which is of fundamental importance 
in probability, statistics, and data science. 
If $Y$ is the data in an experiment,
then $X$ is also data by definition
if $X = \phi(Y)$.
If $Y$ and $\phi$ are measurable, 
then it follows as a consequence that the composition $X = \phi(Y)$ is measurable.
The Doob-Dynkin lemma
\citep[p.603]{DOOB}
\citep[p.7]{Kallenberg02probability}
gives conditions on $X$ and $Y$ that ensures
existence of a $\phi$ such that $X = \phi (Y)$.
In the next section we prove different versions of the Doob-Dynkin lemma, 
and in the final section we briefly discuss the role of the Doob-Dynkin lemma in statistics.
The lemma provides in particular existence and uniqueness of optimal data learning algorithms.

\newpage

\section{The Doob-Dynkin lemma}
\label{sDoobDynkin}

Consider first the case where $X: \Omega \into \Omega_X$ and
$Y: \Omega \into \Omega_Y$ are continuous functions between topological spaces
\citep{KURATOWSKI,KELLEY}.
The following Lemma is probably known in some context,
but I have no reference for this.
A similar comment holds for many other results presented in the following.
\begin{Lemma}[Topological Doob-Dynkin]
  If the image $X(\Omega)$ is a $T_0$ space,
  and $X$ is continuous with respect to the initial topology of $Y$,
  then there exists a unique continuous
  $\phi: Y(\Omega) \into \Omega_X$ such that
  $X = \phi (Y)$.
\end{Lemma}
\begin{proof}
  Let $\omega \in \Omega$, $y = Y(\omega)$, $x = X(\omega)$,
  and define $\phi (y) = x$.
  It must be demonstrated that this gives a well-defined continuous $\phi$.
  Assume that $y = Y(\omega') = y'$.
  It must be proved that
  $x' = X(\omega') = x$.
  Assume for contradiction that $x \neq x'$.
  From $T_0$-separation there exists an open separating $U$.
  Assume without loss of generality that
  $x \in U \not\ni x'$.
  It follows that $\omega' \not\in (X \in U) = (Y \in V)$ so
  $y' \not\in V$ which contradicts $y \in V$ from the assumption $y=y'$.
  Existence of an open $V$ such that $(X \in U) = (Y \in V)$
  follows since $X$ is continuous with respect to the initial
  topology of $Y$.
  This also gives $\phi^{-1} (U) = V \cap Y(\Omega)$ which proves
  continuity of $\phi$.  
\end{proof}
\vekk{
A more general version of the above is given by assuming that
the image $X(\Omega)$ is a $T_0$ space in the relative topology inherited
from the topology of $\Omega_X$.
}
The previous result can also be
proved for more general cases,
including in particular spaces equipped with
the family of co-zero sets from suitable families of
real valued functions \citep{Taraldsen17nonlinProb}.
This includes the case where
$X: \Omega \into \Omega_X$ and
$Y: \Omega \into \Omega_Y$ are measurable functions between measurable spaces
\citep{HALMOS,DUNFORD,RUDIN}.
The simplicity of the following result -
and the fact that it seems to be missing from the standard
presentations linked to conditional expectation
\citep{HALMOS,DOOB,LOEVE,Kallenberg02probability,RaoSwift06probability} -
was part of the original motivation for writing this note.
It should be noted that the function $\phi$ obtained
from the Lemma is only defined on the image $Y(\Omega)$,
and not on the whole set $\Omega_Y$.
The reward is a more general statement - and a simpler proof. 
\begin{Lemma}[Measurable Doob-Dynkin]
  \label{LDoobDynkin}
\vekk{  Assume that
  $X: \Omega \into \Omega_X$ and
  $Y: \Omega \into \Omega_Y$ are measurable where
  the measurable sets in $\Omega_X$ separates points.}
If the image $X(\Omega)$ is $T_1$ and $X$ is measurable with respect to the initial $\sigma$-field of $Y$,
then there exists a unique measurable 
$\phi: Y(\Omega) \into \Omega_X$ such that $X = \phi (Y)$.
\end{Lemma}
\begin{proof}
  The proof is identical with the topological version,
  but the sets $U$ and $V$ in the identity
  $\phi^{-1} (U) = V \cap Y(\Omega)$ are measurable.
  Furthermore, $T_1$ and $T_0$ separation are equivalent since the complement of a measurable set
  is measurable.
\end{proof}

The $T_0$ separation assumption seems to be the natural general assumption
for the proof presented here.
Consideration of the trivial topology $\{\emptyset, \Omega_X\}$ gives as a result that all
functions $X$ are continuous with respect to any given function $Y$,
and in particular with respect to a constant function $Y(\omega) = y_0$.
If $X$ takes at least two values, then it is impossible
to find a $\phi$ such that $X = \phi (Y)$.
A continuation of this argument gives separation assumptions that are not only sufficient,
but also necessary.
The argument in the example holds also for the case of measurable spaces.

Consider next the case where $(\Omega, \cE)$ is a measurable space equipped with
a $\sigma$-finite measure $\pr$:
There are measurable $B_1, B_2, \ldots$ with $\pr (B_k) < \infty$ and
$\Omega = \cup_k B_k $.
This includes the common case of a probability space $\Omega$,
but includes also the case of a Renyi space \citep{RENYI} as needed in a theory of statistics
that includes improper priors
\citep{TaraldsenLindqvist10ImproperPriors,TaraldsenLindqvist16condprob,TaraldsenTuftoLindqvist17improperPosterior}.

The space $(\Omega_Y, \cF_Y)$ is assumed to be a measurable space,
and it becomes a measure space when equipped with the law $\pr_Y$ of
a measurable $Y: \Omega \into \Omega_Y$.
The law is defined by
$\pr_Y (A) = \pr (Y \in A) = \pr \{\omega \st Y(\omega) \in A\}$ for
$A \in \cF_Y$,
and $Y$ is said to be $\sigma$-finite if $\pr_Y$ is $\sigma$-finite.

The initial $\sigma$-field $\cE_Y$ of $Y$ is given
by $\cE_Y = \{Y^{-1} (A) \st A \in \cF_Y\}$.
The following result can be used as a substitute
for the use of the Doob-Dynkin Lemma in the context of conditional expectation,
and represents the second main motivation for writing this note.
It gives an alternative approach to the one
usually followed in standard texts \citep{HALMOS,DOOB,RaoSwift06probability},
and the proof is again much shorter.
It should in particular be observed that the resulting measurable $\phi$
is defined not only on the image $Y(\Omega)$ as in Lemma~\ref{LDoobDynkin},
but on the whole space $\Omega_Y$.
The space $\Omega_X$ from Lemma~\ref{LDoobDynkin} is in the below replaced by the extended real interval
$[0,\infty]$ equipped with the Borel sets generated by the open intervals.
\begin{Lemma}[Conditional expectation Doob-Dynkin]
  \label{LKolmogorov}
  Let $\Gamma: \Omega \into [0,\infty]$ and $Y: \Omega \into \Omega_Y$  be measurable.
  If $Y$ is $\sigma$-finite with initial $\sigma$-field $\cE_Y$,
then there exists
a unique (a.e.) measurable
$\phi: \Omega_Y \into \Omega_\Gamma$ such that
$\E (\Gamma \st \cE_Y) = \phi (Y)$.
\end{Lemma}
\begin{proof}
The Radon-Nikodym theorem \citep[p.121]{RUDIN}
gives a unique (a.e.) $\phi$ such that
$\E (\Gamma \cdot (Y \in A)) = \int_A \phi (y) \, \pr_Y (dy)$
for all measurable $A$ since the left-hand side defines
a measure which is absolutely continuous with respect to
$\pr_Y$ which is assumed to be $\sigma$-finite.
The general change-of-variables theorem gives then
$\E (\Gamma \cdot (Y \in A)) = \int_{(Y \in A)} \phi (Y) \, \pr (d\omega)$
which gives the claim
$\E (\Gamma \st \cE_Y) = \phi (Y)$.
\end{proof}
It should be noted that
\citet[p.53]{KOLMOGOROV} defines the conditional
expectation $\E^y (\Gamma) = \E (\Gamma \st Y=y) = \phi (y)$
directly as given by the above proof.
Later writers, such as \citet[p.17-18]{DOOB},
defines first a conditional expectation $\E (\Gamma \st \cF)$,
and then $\E (\Gamma \st Y) = \E (\Gamma \st \cE_Y)$ as a special case.
The Doob-Dynkin Lemma is then needed to finally define
$\E^y (\Gamma) = \E (\Gamma \st Y=y) = \phi (y)$.
The advantage of the original approach of Kolmogorov is that,
as in the above proof,
existence and uniqueness (a.e.) of  $\phi(y) = \E^y (\Gamma)$
is proved directly without having to refer to a Doob-Dynkin type Lemma.

The result in Lemma~\ref{LKolmogorov} is generalized directly to the setting where
$\Omega_X = [0,\infty]$ is replaced by a separable Banach space
$\Omega_X$ by duality and decomposition of a complex number in four
unique components in $[0,\infty)$.
An even more general version follows also as a consequence by an alternative argument.
\begin{Lemma}[a.e. Doob-Dynkin]
  \label{LDoobDynkinSigma}
Let $\Omega$ be a $\sigma$-finite measure space.
If $X(\Omega)$ is contained in a standard Borel space
and $X$ is measurable with respect to the initial $\sigma$-field of
a $\sigma$-finite $Y$,
then there exists a unique (a.e.) measurable 
$\phi: \Omega_Y \into \Omega_X$ such that $X = \phi (Y)$.
\end{Lemma}
\begin{proof}
The characterization theorem of standard Borel spaces \citep[p.90]{KECHRIS}
shows that it is sufficient to consider $\Omega_X = [0,\infty]$.
The assumptions and Lemma~\ref{LKolmogorov}
give $X = \E (X \st \cE_Y) = \phi (Y)$.
\end{proof}
A standard Borel space is a space equipped with
the $\sigma$-field of a separable complete metric space.
A measurable set of a standard Borel space is also a standard Borel space
\citep[p.75]{KECHRIS}.
The previous Lemma can be generalized by an alternative proof
which is essentially the proof given by \citet[p.603]{DOOB}
for a less general statement.
\begin{Lemma}[Standard Doob-Dynkin]
  \label{LDoobDynkinStandard}
If $X(\Omega)$ is contained in a standard Borel space
and $X$ is measurable with respect to the initial $\sigma$-field of
a measurable $Y$,
then there exists a measurable 
$\phi: \Omega_Y \into \Omega_X$ such that $X = \phi (Y)$.
\end{Lemma}
\begin{proof}
Assume first that $X (\Omega) = \{x_1, x_2, \ldots\}$ is $T_1$
in the relative $\sigma$-field from $\Omega_X$.
The separation assumption gives a countable partition 
$\Omega_X = \uplus_k U_k$ such that $U_k \cap X(\Omega) = \{x_k\}$.
The $\cE_Y$ measurable $X$ gives then $V_k$ such that
$(X \in U_k) = (Y \in V_k)$.
The partition $\Omega = \uplus_k (X \in U_k)$
corresponds to a partition $Y(\Omega) = \uplus_k (Y(\Omega) \cap V_k)$.
Each $V_k$ can be replaced by $V_k \setminus (\cup_{l \neq k} V_l)$
so it can be assumed that the $V_1, V_2, \ldots$ are disjoint
and then that $\Omega_Y = \uplus_k V_k$.
The required $\phi$ can finally be defined by $\phi (y) = x_k$ for
$y \in V_k$.
  
The characterization theorem of standard Borel spaces \citep[p.90]{KECHRIS}
shows that it is sufficient to consider $\Omega_X = [0,\infty]$ for the general case.
For this case it follows then that
$X = \lim_n X_n$ for a monotone increasing sequence of
simple functions $X_n$ that are all $Y$ measurable.
The above argument gives $\phi_n$ such that
$X_n = \phi_n (Y)$,
and $\phi (y) = \sup_n \phi_n (y)$ gives the claim.
\end{proof}
Lemma~\ref{LDoobDynkin} does not provide a measurable
$\phi: \Omega_Y \into \Omega_X$ since
the image $Y(\Omega)$ may fail to be measurable.
If, however, it is assumed that $X(\Omega)$ is contained in a standard Borel space,
then the Kuratowski extension theorem \citet[p.73]{KECHRIS}
ensures that there exists a measurable extension
$\tilde{\phi} : \Omega_Y \into \Omega_X$ of $\phi$.
It follows hence that Lemma~\ref{LDoobDynkin} combined with the
Kuratowski extension theorem gives an alternative proof of Lemma~\ref{LDoobDynkinStandard}.
Alternatively, Lemma~\ref{LDoobDynkinStandard}, can be used to obtain a proof
of the Kuratowski extension theorem.

A natural question next:
Is it possible to generalize Lemma~\ref{LDoobDynkinSigma} by relaxing the
conditions on $X(\Omega)$?
This would then also give an alternative to the Kuratowski extension
theorem for the case where $\Omega_Y$ is a $\sigma$-finite measure space.
Lemma~\ref{LDoobDynkinSigma} provides
a $\phi^* : \Omega_Y \into \Omega_X$ that extends the unique
$\phi: Y(\Omega) \into \Omega_X$ in the sense that
$\phi^* = \phi$ almost everywhere on $Y(\Omega)$.
This is a weaker result than the Kuratowski extension theorem.
The following argument gives, unfortunately,
only an alternative proof of Lemma~\ref{LDoobDynkinSigma}.

Assume that $X(\Omega)$ is contained in a space that contains a 
family $\psi_1, \psi_2, \ldots$ of indicator functions that separates
points and generates the $\sigma$ field.
Lemma~\ref{LKolmogorov} gives $\psi_k (X) = \E (\psi_k (X) \st Y) = \phi_k (Y)$
and this determines uniqueness of a measurable $\phi^* (y)$ from
$\psi_k (\phi^* (y)) = \phi_k (y)$ for all $k$ for
$y \in \cap_k D (\phi_k) = D(\phi^*)$.
The $D(\phi_k)$ are chosen such that
$\phi = \phi^*$ on $Y(\Omega) \cap D(\phi^*)$ and
$\pr_Y (D(\phi^*)^c) = 0$.
Unfortunately,
completeness is here needed to ensure existence of $\phi^* (y)$,
and the result is hence only an alternative proof of
Lemma~\ref{LDoobDynkinSigma}.

An alternative attempt is to consider a generalization of
the Kuratowski extension theorem via an extension of
Lemma~\ref{LKolmogorov} to the case of
a possibly non-separable Hilbert space.
It gives a measurable $<f,\phi^* (y)>$,
but the  problem is that the good domain of $\phi^*$ will depend on $f$.
It is only separability that  ensures existence of a
countable family of vectors $f$ that can determine $\phi^*$ on a good domain $D^*$:
$\phi = \phi^*$ on $Y(\Omega) \cap D^*$ where
$\pr_Y (\Omega_Y \setminus D^*) = 0$.
The conclusion is that neither the completeness nor the separability assumptions are easily removed
even when relaxing the requirements in the Kuratowski extension theorem
into an almost everywhere statement.
It is possible that a version can be obtained by completing the
$\sigma$-field $\cF_Y$, but we leave this question open.

\section{Optimal learning from data}
\label{sOptimal}

A statistical model is given by the structure \citep{TaraldsenTuftoLindqvist17improperPosterior}
%
\be{StatModFirst}
\scalebox{1.2}{
\begin{tikzcd}[row sep=normal, ampersand replacement=\&]
\&\Omega_\Theta \arrow[r, "\psi"] \& \Omega_\Gamma  \\
(\Omega, {\cal E}, \pr) \arrow[ur, "\Theta"] 
\arrow[drr, "X" near end] \arrow[dr, "Y"'] 
\arrow[urr, "\Gamma"' near end] \& \& \\
\&\Omega_Y \arrow[r, "\phi"'] \& \Omega_X
\end{tikzcd}
}
\ee
The uncertainty is modeled by the law $\pr$ on
the space $\Omega$ from which an unknown $\omega$ has been drawn.
The model data $y = Y (\omega)$ is observed and the aim is to determine
$\phi (y)$ such that this gives optimal learning
about the focus parameter $\gamma = \Gamma (\omega) = \psi (\theta)$
where $\theta = \Theta (\omega)$ is the unknown model parameter.

Bayesian analysis is given by assuming that
the prior law $\pr_\Theta$ and the conditional data
distribution $\pr_Y^\theta$ is specified,
or more generally that the joint law $\pr_{Y,\Theta}$ is specified.
Optimal learning in the sense of estimating $\gamma$
can be defined by attempting to find
an optimal action $X (\omega) = x = \phi (y)$
that minimizes the Bayes risk \citep[p.11]{BERGER}
\be{qloss}
r = \E \norm{\Gamma - X}^2 = \E \norm{\psi(\Theta) - \phi (Y)}^2
\ee
where it is assumed that $\Omega_X = \Omega_\Gamma$ is a separable Hilbert space.
The assumption $X = \phi (Y)$ means in particular that
$X$ is measurable with respect to the initial $\sigma$-field $\cE_Y$ of $Y$.
If $Y$ is $\sigma$-finite,
then it follows that $L^2 (\cE_Y)$ is a closed subspace of $L^2 (\cE)$
and the projection
\be{eqproj}
X = \E (\Gamma \st \cE_Y) = \phi (Y)
\ee
is the minimizer of the Bayes risk.
Existence of a required $\phi$ follows from Lemma~\ref{LDoobDynkin},
but the other versions of the Doob-Dynkin Lemma can also be used.
It should be observed here that the argument is more general than usual
since the probability space of Kolmogorov has been replaced by
a Renyi space $\Omega$.

The previous includes also the case of the Kalman-Bucy filter as
described in more detail by \citet[p.81-108]{OKSENDAL5th}.
The unknown parameter is then $\Gamma = \Gamma_t$ at a given time $t$,
and the data $Y$ is the observations $Y_s$ for all $0 \le s \le t$
of a stochastic process which is a filtered and noisy version of
$\Gamma_s$ for $0 \le s \le t$.
The optimal solution is again given
by equation~(\ref{eqproj}) \citep[p.83, Theorem 6.1.2]{OKSENDAL5th}. 
The actual calculation for the 1-dimensional Kalman-Bucy filter
involves solving a nonlinear ordinary differential equation which
gives the coefficients of a stochastic differential equation
that determines the solution based on the observations \citep[p.96, Theorem 6.2.8]{OKSENDAL5th}.

The main reason for mentioning the Kalman-Bucy filter is that
it corresponds to a case where both
the model parameter space $\Omega_\Theta$
and the the data space $\Omega_Y$ are infinite dimensional.
They can both in this application be identified with the set of continuous
paths indexed with a time parameter 
\citep[p.22]{OKSENDAL5th},
but for some applications it is more appropriate to use a space
of tempered distributions.
The concept of a random tempered distribution can be further
generalized using the ideas of \citet{Skorohod84}
for strong linear random operators which generalizes
the concept of a random operator.

Consider again the Bayes risk in equation~(\ref{qloss}).
If it is assumed that $Y$ is $\sigma$-finite,
then the following decomposition holds
\be{Yqloss}
r = \E (\E (\norm{\psi(\Theta) - \phi (Y)}^2 \st Y))
= \int r^y \, \pr_Y (dy)
\ee
It follows that the Bayes risk is minimized if the Bayes posterior risk
\be{BPqloss}
r^y = \E (\norm{\psi(\Theta) - \phi (y)}^2 \st Y=y)
\ee
is minimized for each $y$.
This gives the explicit solution
\be{BPqlossEst}
\phi (y) = \E^y \psi(\Theta)
\ee
It should be observed that the Bayes posterior risk can be minimized and
uniformly finite even in cases where the Bayes risk in 
equation~(\ref{qloss}) is infinite.
Minimization of the Bayes posterior loss is hence a more generally
applicable procedure for determining a decision rule
$\phi$ that gives optimal learning.

A simple example is given by $y = \theta + u$ where
$u$ is drawn from a standard one dimensional normal distribution.
If the prior for $\theta$ is Lebesgue measure on the real line,
then the posterior for $\theta$ equals
the fiducial distribution $\theta = y - u$:
The posterior equals a $\normvar (y, 1)$ distribution.
Consider the simplest case where $\gamma = \psi (\theta) = \theta$,
which gives $\phi (y) = y$ from equation~(\ref{BPqlossEst}).
The Bayes posterior risk is then
$r^y = 1$ from equation~(\ref{BPqloss}),
and the Bayes risk $r = \infty$.
The latter follows since the marginal law of $y$ is also
Lebesgue measure on the real line.

This example can be generalized to a general location problem,
including general linear regression,
and even more general kinds of group models.
In the case of an infinite dimensional Hilbert space
$\Omega_Y$ the invariant measure does not exist,
but for this case the fiducial posterior loss can be used
as a substitute for the Bayes posterior loss,
and gives optimal frequentist inference
\citep{TaraldsenLindqvist13fidopt}.

Optimal frequentist inference can be defined as
given by a $\phi$ that minimizes the frequentist risk
\be{BFqloss}
r^\theta = \E^\theta (\norm{\phi (Y) - \gamma}^2) =
\E (\norm{\psi(\theta) - \phi (Y)}^2 \st \Theta = \theta )
\ee
uniformly for each model parameter $\theta$.
The quadratic loss function is here used for simplicity,
and many alternatives exist depending on the kind of inference in particular problems.
Restrictions on the class of allowable functions $\phi$ are commonly given
by demanding unbiasedness or equivariance with respect  to a group action 
\citep{TaraldsenLindqvist13fidopt}.
It follows that an optimal frequentist action $\phi$, if it exists,
will also minimize the Bayes risk since a $\sigma$-finite $\Theta$ ensures
\be{BFqloss}
r = \int r^\theta \, \pr_\Theta (d\theta)
\ee
In many cases, however, there exists no optimal frequentist action.
A good alternative is often given by the optimal Bayesian posterior action
as can be inferred from the previous arguments.
The prior $\pr_\Theta$ is then chosen not based on prior knowledge,
but so that it gives appropriate weight to regions in the model parameter space
that are considered important.

\vekk{
In conventional theory \citep{SCHERVISH} the space $\Omega$
is a probability space.
In the more general setting of a Rényi space $\Omega$ considered here
the underlying law $\pr$ is a conditional probability law with
a corresponding bunch $\cB \subset \cE$.
The law of the data $X$ given the parameter $\Theta = \theta$
is defined by $\pr_X^\theta (A) = \pr (X \in A \st \Theta = \theta)$.
The law of the model parameter $\Theta$ given the data $X=x$
is defined by $\pr_\Theta^x (A) = \pr (\Theta \in A \st X = x)$.
The posterior law $\pr_\Gamma^x$ of a parameter $\gamma = \psi (\theta)$
and the law $\pr_Y^\theta$ of a statistic are determined by this
and equation~(\ref{eqStatModFirst}).

\section{Discussion}
\label{sDisc}

Alternative proofs!

\citet[p.613]{SCHERVISH}

Bayes and updating

Optimal estimator and conditioning

Kolmogorov, QM, Aarnes, Dempster-Shafer

Rain tomorrow a'la Renyi versus Lindley.
Versus coin tossing

\section{Conclusion}

}

\bibliography{bib,gtaralds}
\bibliographystyle{imsart-nameyear}

\end{document}